\documentclass[11pt, final]{article}

\usepackage[dvips]{graphicx}
\usepackage{graphics}

\usepackage{amsthm, amssymb, amsmath, amsfonts}
 \newtheorem{theorem}{Theorem}
 \newtheorem{corollary}[theorem]{Corollary}
 \newtheorem{lemma}[theorem]{Lemma}

 \newtheorem{definition}[theorem]{Definition}
 \newtheorem{remark}[theorem]{Remark}
 
\newcommand{\toto}{\rightrightarrows}
\newcommand{\R}{{\mathbb R}}
\newcommand{\oR}{{\overline{\mathbb R}}}

\newcommand{\To}{\longrightarrow}
\newcommand{\p}{\partial}
\def\1{\^{\i}}
\def\2{\u{a}}
\def\3{\c{s}}
\def\4{\^{a}}
\def\5{\c{t}}

\def\b{\beta}
\def\e{\epsilon}

\def\g{\gamma}
\def\l{\lambda}
\def\<{\langle}
\def\>{\rangle}

\DeclareMathOperator*\dom{dom}

\DeclareMathOperator*\id{Id}

\DeclareMathOperator*\argmin{argmin}
\DeclareMathOperator*\prox{prox}
\DeclareMathOperator*\ran{ran}%
\DeclareMathOperator*\zer{zer}
\DeclareMathOperator*\gr{Gr}%
\begin{document}



              
\title{Second order dynamical systems with penalty terms associated to monotone inclusions}

\author{Radu Ioan Bo\c{t} \thanks{University of Vienna, Faculty of Mathematics, Oskar-Morgenstern-Platz 1, A-1090 Vienna, Austria, 
email: radu.bot@univie.ac.at. Research partially supported by FWF (Austrian Science Fund), project I 2419-N32.} \and 
Ern\"{o} Robert Csetnek \thanks {University of Vienna, Faculty of Mathematics, Oskar-Morgenstern-Platz 1, A-1090 Vienna, Austria, 
email: ernoe.robert.csetnek@univie.ac.at. Research supported by FWF (Austrian Science Fund), Lise Meitner Programme, project M 1682-N25.} 
 \and Szil\'{a}rd Csaba L\'{a}szl\'{o} \thanks{Technical University of Cluj-Napoca, Department of Mathematics, 
 Str. Memorandumului nr. 28, 400114 Cluj-Napoca, Romania, e-mail: laszlosziszi@yahoo.com}}              
\maketitle

\noindent \textbf{Abstract.} In this paper we investigate in a Hilbert space setting a second order dynamical system of the form 
$$\ddot{x}(t)+\g(t)\dot{x}(t)+x(t)-J_{\l(t) A}\big(x(t)-\l(t) D(x(t))-\l(t)\b(t)B(x(t))\big)=0,$$
where $A:{\mathcal H}\toto{\mathcal H}$ is a maximal monotone operator, $J_{\l(t) A}:{\mathcal H}\To{\mathcal H}$ is the resolvent operator of $\l(t)A$ and $D,B: {\mathcal H}\rightarrow{\mathcal H}$ are cocoercive operators, and $\lambda,\beta :[0,+\infty)\rightarrow (0,+\infty)$, and  $\gamma:[0,+\infty)\rightarrow (0,+\infty)$ are step size, penalization and, respectively, damping functions, all depending on time. We show the existence and uniqueness of
strong global solutions in the framework of the Cauchy-Lipschitz-Picard Theorem and prove ergodic asymptotic convergence for 
the generated trajectories to a zero of the operator $A+D+{N}_C,$ where $C=\zer B$ and $N_C$ denotes the normal cone operator of $C$. To this end we use Lyapunov analysis combined with the celebrated Opial Lemma in its ergodic continuous version. Furthermore, 
we show strong convergence for trajectories to the unique zero of $A+D+{N}_C$, provided that $A$ is a strongly monotone operator. 
\vspace{1ex}

 \vspace{1ex}


\section{Introduction and preliminaries}\label{sec1}

Consider the bilevel optimization problem
\begin{equation}\label{e1}
\inf_{x\in \argmin \psi} \{f(x)+g(x)\},
\end{equation}
where $f :{\mathcal H}\To\R\cup\{+\infty\}$ is a proper convex and lower semicontinuous function, 
$g,\psi:{\mathcal H}\To\R$ are convex and (Fr\'echet) differentiable functions both with Lipschitz continuous gradients, and $\argmin \psi$ denotes the set of global minimizers of $\psi$, assumed to be nonempty.

By making use of the indicator function of $\argmin \psi$, $\eqref{e1}$ can be rewritten as
\begin{equation}\label{e2}
\inf_{x\in {\mathcal H}} \{f(x)+g(x)+\delta_{ \argmin \psi}(x)\}.
\end{equation}
Obviously, $x\in\argmin\psi$ is an optimal solution of $(\ref{e2})$ if and only if 
$0\in \p(f+g+\delta_{ \argmin \psi})(x)$, which can be split in
\begin{equation}\label{e3}
0\in \p f(x)+\nabla g(x)+\p\delta_{ \argmin \psi}(x),
\end{equation}
provided a suitable qualification condition which guarantees the subdifferential sum rule holds.

Using that $\p\delta_{ \argmin \psi}(x)={N}_{\argmin \psi}(x)$ and $\argmin\psi=\zer \nabla \psi$,  
$\eqref{e3}$  is nothing else than
\begin{equation}\label{e31}
0\in \p f(x)+\nabla g(x)+N_{ \zer \nabla \psi}(x).
\end{equation}

This motivates us to investigate the following inclusion problem
\begin{equation}\label{e4}
0\in Ax+Dx+{N}_C(x),
\end{equation}
where $A:{\mathcal H}\toto {\mathcal H}$ is a maximally monotone operator, $D:{\mathcal H}\To {\mathcal H}$ is a  $L_D^{-1}$-cocoercive operator, 
and $B:{\mathcal H}\To {\mathcal H}$ is a $L_B^{-1}$-cocoercive operator with $L_D, L_B>0$, $C=\zer B$, and $N_C$ denotes the normal cone operator 
of the set $C$. We recall that by the classical Baillon-Haddad Theorem, the gradient of a convex and (Fr\'echet) differentiable function is 
$L$-Lipschitz continuous, for $L>0$, if and only if it is $L^{-1}$-cocoercive, see for instance \cite[Corollary 18.16]{bauschke-book}). 

In \cite{BCs}, a first order dynamical system has been assigned to the monotone inclusion $(\ref{e4})$, and it has been shown that the generated trajectories converge to a solution of it. In this paper, we 
assign to $(\ref{e4})$ the following second order dynamical system 
\begin{equation}\label{e5}
\left\{
\begin{array}{ll}
\ddot{x}(t)+\g(t)\dot{x}(t)+x(t)=J_{\l(t) A}\big(x(t)-\l(t) D(x(t))-\l(t)\b(t)B(x(t))\big)\\
x(0)=u_0,\,\dot{x}(0)=v_0,
\end{array}
\right.
\end{equation}
where $u_0,v_0\in {\mathcal H}$ and $\g,\l,\b:[0,+\infty)\To(0,+\infty).$ Dynamical systems 
governed by resolvents of maximally monotone operators have been considered in \cite{abbas-att-arx14, abbas-att-sv}, and then
further developed in \cite{b-c-dyn-KM, b-c-dyn-sec-ord}. 

The study of second order dynamical systems is motivated by the fact that the presence of the acceleration term $\ddot x(t)$ can lead to 
better convergence properties of the trajectories. Time discretizations of second order dynamical systems give usually rise to numerical 
algorithms with inertial terms which have been shown to have improved convergence properties (see \cite{Ne}). The geometric damping function 
$\gamma$ which acts on the velocity can in some situations accelerate the 
asymptotic properties of the orbits, as emphasized for example in \cite{su-boyd-candes}. 

For $B=0$ and $\lambda$ is constant, the differential equation \eqref{e5} becomes the second order forward-backward dynamical system investigated in \cite{b-c-dyn-sec-ord} in relation to the monotone inclusion problem
$$0\in Ax+Dx.$$
On the other hand, when particularized to the monotone inclusion system \eqref{e31}, the differential equation \eqref{e5} reads
\begin{equation}\label{e6}
\left\{
\begin{array}{ll}
\ddot{x}(t)+\g(t)\dot{x}(t)+x(t)=\prox_{\l(t) f}\big(x(t)-\l(t) \nabla g(x(t))-\l(t)\b(t)\nabla \psi(x(t))\big)\\
x(0)=u_0,\,\dot{x}(0)=v_0,
\end{array}
\right.
\end{equation}
where we made used of the fact that the resolvent of the subdifferential of a proper, convex and lower semicontinuous function is the proximal point operator of the latter. In case $f=0$, $\gamma$ is constant and $\lambda$ is also constant and identical to $1$, \eqref{e6} leads to the differential equation that has been investigated in \cite{att-cza-16} and \cite{b-c-aa}. 

The first part of the paper is devoted to the proof of the existence and uniqueness of (locally) absolutely continuous trajectories generated by the dynamical system \eqref{e5}; an important ingredient for this analysis is the Cauchy-Lipschitz-Picard Theorem (see \cite{haraux, sontag}). The proof of the convergence of the trajectories to a solution of \eqref{e4} is the main result of the manuscript. Provided that a condition expressed in terms of the Fitzpatrick function of the cocoercive operator $B$ is fulfilled, 
we prove weak ergodic convergence of the orbits.  Furthermore, we show that, if the operator $A$ is strongly monotone, then one obtains even strong (non-ergodic) convergence for the generated trajectories.

In the remaining of this section, we explain the notations we used up to this point and will use throughout the paper (see \cite{bo-van, bauschke-book, simons}).

The real Hilbert space ${\mathcal H}$ is endowed with \textit{inner product} $\langle\cdot,\cdot\rangle$ and associated \textit{norm}
$\|\cdot\|=\sqrt{\langle \cdot,\cdot\rangle}$. The \textit{normal cone} of of a set $S\subseteq {\mathcal H}$ is defined by $N_S(x)=\{u\in {\mathcal H}:\langle y-x,u\rangle\leq 0 \ \forall y\in S\}$, if $x\in S$ and $N_S(x)=\emptyset$
for $x\notin S$. The following characterization of the elements of the normal cone of a nonempty set by means of its support function will 
be used several times in the paper: for $x\in S$, $u\in N_S(x)$ if and only if $\sigma_S(u)=\langle x,u\rangle$, where $\sigma_S : {\mathcal H} \rightarrow \R\cup\{+\infty\}$ is defined by $\sigma_S(u)=\sup_{y\in S}\langle y,u\rangle$.

Let $A:{\mathcal H}\rightrightarrows {\mathcal H}$ be a set-valued operator. We denote by
$\gr A=\{(x,u)\in {\mathcal H}\times {\mathcal H}:u\in Ax\}$ its \emph{graph}, by $\dom A=\{x \in {\mathcal H} : Ax \neq \emptyset\}$ its
\emph{domain} and by $\ran A=\{u\in {\mathcal H}: \exists x\in {\mathcal H} \mbox{ s.t. }u\in Ax\}$ its \emph{range}. The notation $\zer A=\{x\in {\mathcal H}: 0\in Ax\}$ stands for the \textit{set of zeros} of the operator $A$. 
We say that $A$ is \emph{monotone} if $\langle x-y,u-v\rangle\geq 0$ for all $(x,u),(y,v)\in\gr A$. Further, a monotone operator $A$ is said to be
\emph{maximally monotone}, if there exists no proper monotone extension of the graph of $A$ on ${\mathcal H}\times {\mathcal H}$. 
The following characterization of the zeros of a maximally monotone operator will be crucial in the 
asymptotic analysis of \eqref{e5}: if $A$ is maximally monotone, then 
\begin{equation*}
z\in\zer A \mbox{ if and only if }\langle u-z,w\rangle\geq 0\mbox{ for all }(u,w)\in \gr A.
\end{equation*}

The \emph{resolvent} of $A$, $J_A:{\mathcal H} \rightrightarrows {\mathcal H}$, is defined by $p\in J_A(x)$ if and only if 
$x\in p+Ap$. Moreover, if $A$ is maximally monotone, then $J_A:{\mathcal H} \rightarrow {\mathcal H}$ is single-valued and maximally monotone
(cf. \cite[Proposition 23.7 and Corollary 23.10]{bauschke-book}). We will also use the \emph{Yosida approximation} of the operator $A$,
which is defined for $\alpha>0$ by $A_{\alpha}=\frac{1}{\alpha}(\id -J_{\alpha A})$,  where
$\id :{\mathcal H} \rightarrow {\mathcal H}, \id(x) = x$ for all $x \in {\mathcal H}$, is the \textit{identity operator} on ${\mathcal H}$.

The notion of \emph{Fitzpatrick function} associated to a monotone operator $A$ will be important in the formulation of the condition 
under which the convergence of the trajectories is achieved. It is defined as
$$\varphi_A:{\mathcal H}\times {\mathcal H}\rightarrow \oR, \ \varphi_A(x,u)=\sup_{(y,v)\in\gr A}\{\langle x,v\rangle+\langle y,u\rangle-\langle y,v\rangle\},$$
and it is a convex and lower semicontinuous function. Introduced by Fitzpatrick in \cite{fitz},
this notion played in the last years a crucial role in the investigation of maximality of monotone operators by means of convex analysis specific tools
(see \cite{bauschke-book, bausch-m-s, bo-van, simons} and the references therein).
We notice that, if $A$ is maximally monotone, then $\varphi_A$ is proper and 
$$\varphi_A(x,u)\geq \langle x,u\rangle \ \forall (x,u)\in {\mathcal H}\times {\mathcal H},$$
with equality if and only if $(x,u)\in\gr A$. We refer the reader to \cite{bausch-m-s} for explicit formulae of Fitzpatrick functions associated to particular classes of monotone operators.

Let $\gamma>0$ be arbitrary. A single-valued operator $A:{\mathcal H}\rightarrow {\mathcal H}$ is said to be \textit{$\gamma$-cocoercive}, if
$\langle x-y,Ax-Ay\rangle\geq \gamma\|Ax-Ay\|^2$ for all $(x,y)\in {\mathcal H}\times {\mathcal H}$, and \textit{$\gamma$-Lipschitz continuous},
if $\|Ax-Ay\|\leq \gamma\|x-y\|$ for all $(x,y)\in {\mathcal H}\times {\mathcal H}$.

For a proper, convex and lower semicontinuous function $f:{\mathcal H}\rightarrow\R\cup\{+\infty\}$, its (convex) subdifferential at $x\in {\mathcal H}$ is defined as
$$\partial f(x)=\{u\in {\mathcal H}:f(y)\geq f(x)+\<u,y-x\> \ \forall y\in {\mathcal H}\}.$$ When seen as a set-valued mapping, it is a
maximally monotone operator  and its resolvent is given by $J_{\gamma \partial f}=\prox_{\gamma f}$ (see \cite{bauschke-book}),
where $\prox_{\gamma f}:{\mathcal H}\rightarrow {\mathcal H}$,
\begin{equation}\label{prox-def}\prox\nolimits_{\gamma f}(x)=\argmin_{y\in {\mathcal H}}\left \{f(y)+\frac{1}{2\gamma}\|y-x\|^2\right\},
\end{equation}
denotes the proximal point operator of $f$.

\section{Existence and uniqueness of the trajectory}\label{sec2}

We start by specifying which type of solutions are we considering in the analysis of the dynamical system \eqref{e5}.  

\begin{definition}\label{str-sol}\rm We say that $x:[0,+\infty)\rightarrow {{\mathcal H}}$ is a strong global solution of \eqref{e5}, if the
following properties are satisfied:

(i) $x,\dot x:[0,+\infty)\rightarrow {{\mathcal H}}$ are locally absolutely continuous, in other words, absolutely continuous on each interval $[0,b]$ for $0<b<+\infty$;

(ii) $\ddot{x}(t)+\g(t)\dot{x}(t)+\big(x(t)-J_{\l(t) A}(x(t)-\l(t) D(x(t))-\l(t)\b(t)B(x(t)))\big)=0$ for almost every $t\geq0$;

(iii) $x(0)=u_0$ and $\dot x(0)=v_0$.
\end{definition}

For proving existence and uniqueness of the strong global solutions of \eqref{e5}, we use the Cauchy-Lipschitz-Picard Theorem for absolutely continues trajectories (see for example
\cite[Proposition 6.2.1]{haraux}, \cite[Theorem 54]{sontag}). The key argument is that one can rewrite \eqref{e5} as a particular first order dynamical system in a suitably chosen product space (see also \cite{alv-att-bolte-red}).

To this end we make the following assumption:
$$(H1):\,\, \g, \l,\b:[0,+\infty)\To(0,+\infty)\mbox{ are continuous on each interval } [0,b],\mbox{ for }0<b<+\infty,$$
which also describes the framework in which we will carry out the convergence analysis in the forthcoming sections.

\begin{theorem}\label{t1} Suppose that $\g,\l$ and $\b$ satisfy $(H1)$. Then for every $u_0,v_0\in {\mathcal H}$ there exists a unique strong global solution of $(\ref{e5}).$
\end{theorem}
\begin{proof}
Define $X:[0,+\infty)\To {\mathcal H}\times {\mathcal H}$ as $X(t)=(x(t),\dot{x}(t)).$ Then (\ref{e5}) is equivalent to
\begin{equation}\label{firstorder}
\left\{
\begin{array}{ll}
\dot{X}(t)=F(t, X(t))\\
X(0)=(u_0,v_0),
\end{array}
\right.
\end{equation}
where $F(t,u,v)=\big(v,-\g(t)v-u+J_{\l(t) A}(u-\l(t) D(u)-\l(t)\b(t)B(u))\big).$

First we show that $F(t,\cdot,\cdot)$ is Lipschitz continuous with a Lipschitz constant $L(t)\in L^1_{loc}([0,+\infty))$, for every $t\geq 0$. Indeed, 
\begin{align*}
\|F(t,u,v)-F(t,\bar{u},\bar{v})\|= & \ \sqrt{\|v-\bar{v}\|^2+\|\g(t)(\bar{v}-v)+(\bar{u}-u)+(J_{\l(t) A}(s)-J_{\l(t) A}(\bar{s}))\|^2} \\
\le & \ \sqrt{\|v-\bar{v}\|^2+2\|\g(t)(\bar{v}-v)+(\bar{u}-u)\|^2+2\|J_{\l(t) A}(s)-J_{\l(t) A}(\bar{s})\|^2} \\
\le & \ \sqrt{(1+4\g^2(t))\|v-\bar{v}\|^2+4\|(\bar{u}-u)\|^2+2\|J_{\l(t) A}(s)-J_{\l(t) A}(\bar{s})\|^2},
\end{align*}
where $s=u-\l(t) D(u)-\l(t)\b(t)B(u)$ and $\bar{s}=\bar{u}-\l(t) D(\bar{u})-\l(t)\b(t)B(\bar{u}).$

By using the nonexpansivity of $J_{\l(t) A}$ we get
\begin{align*}
\|J_{\l(t) A}(s)-J_{\l(t) A}(\bar{s})\|\le & \ \|(u-\bar{u})+\l(t) (D(\bar{u})-D(u))+\l(t)\b(t)(B(\bar{u})-B(u)\|\\
\le & \ (1+\l(t)L_D+\l(t)\b(t)L_B)\|u-\bar{u}\|.
\end{align*}
Hence,
\begin{align*}
\|F(t,u,v)-F(t,\bar{u},\bar{v})\| \le & \ \sqrt{(1+4\g^2(t))\|v-\bar{v}\|^2+(4+2(1+\l(t)L_D+\l(t)\b(t)L_B)^2)\|u-\bar{u}\|^2}\\
\le & \ \sqrt{5+4\g^2(t)+2(1+\l(t)L_D+\l(t)\b(t)L_B)^2}\sqrt{\|u-\bar{u}\|^2+\|v-\bar{v}\|^2} \\
\le & \ (\sqrt{5}+2\g(t)+\sqrt{2}(1+\l(t)L_D+\l(t)\b(t)L_B))\|(u,v)-(\bar{u},\bar{v})\|.
\end{align*}
Since $\g,\l,\b\in L^1_{loc}([0,+\infty))$, it follows that  $$L(t):=\sqrt{5}+2\g(t)+\sqrt{2}(1+\l(t)L_D+\l(t)\b(t)L_B)$$
is also locally integrable on $[0,+\infty)$.

Next we show that $F(\cdot,u,v)\in L^1_{loc}([0,+\infty),{\mathcal H}\times {\mathcal H})$ for all $u,v\in {\mathcal H}.$ We fix $u,v\in {\mathcal H}$ and $b>0$, and notice that
\begin{align*}
\int_0^b\|F(t,u,v)\|dt= & \ \int_0^b\sqrt{\|v\|^2+\|\g(t)v+u-J_{\l(t) A}(u-\l(t) D(u)-\l(t)\b(t)B(u))\|^2}dt\\
\le & \int_0^b\sqrt{(1+2\g^2(t))\|v\|^2+4\|u\|^2+4\|J_{\l(t) A}(u-\l(t) D(u)-\l(t)\b(t)B(u))\|^2}dt.
\end{align*}

According to $(H1)$, there exist positive numbers $\underline{\l}$ and $\underline{\b}$ such that $0<\underline{\l}\le\l(t)$ and $0<\underline{\b}\le\b(t)$ for all $t\in[0,b]$. Hence
\begin{align*}
\|J_{\l(t) A}(u-\l(t) D(u)-\l(t)\b(t)B(u))\| & =\\
\|J_{\l(t) A}(u-\l(t) D(u)-\l(t)\b(t)B(u))-J_{\l(t) A}(u-\underline{\l} D(u)-\underline{\l}\underline{\b}B(u))+J_{\l(t) A}(u-\underline{\l} D(u)-\underline{\l}\underline{\b}B(u))\| & \le\\
(\l(t)-\underline{\l})\|D(u)\|+(\l(t)\b(t)-\underline{\l}\underline{\b})\|B(u)\|+\|J_{\l(t) A}(u-\underline{\l} D(u)-\underline{\l}\underline{\b}B(u))\|.&
\end{align*}
In addition,
\begin{align*}
 \|J_{\l(t) A}(u-\underline{\l} D(u)-\underline{\l}\underline{\b}B(u))\| & =\\
\|J_{\l(t) A}(u-\underline{\l} D(u)-\underline{\l}\underline{\b}B(u))-J_{\underline{\l} A}(u-\underline{\l} D(u)-\underline{\l}\underline{\b}B(u))+J_{\underline{\l} A}(u-\underline{\l} D(u)-\underline{\l}\underline{\b}B(u))\| & \le\\
(\l(t)-\underline{\l})\|A_{\underline{\l}}(u-\underline{\l} D(u)-\underline{\l}\underline{\b}B(u))\|+\|J_{\underline{\l} A}(u-\underline{\l} D(u)-\underline{\l}\underline{\b}B(u))\|,
\end{align*}
where the last inequality follows from the Lipschitz property of the resolvent operator as a function of the step size, 
which basically follows by combining \cite[Proposition 2.6]{brezis} and \cite[Proposition 23.28]{bauschke-book} (see also \cite[Proposition 3.1]{abbas-att-sv}).
Hence,
\begin{align*}
\int_0^b\|F(t,u,v)\|dt\le & \ \int_0^b\left((1+\sqrt{2}\g(t))\|v\|+2\|u\|+2(\l(t)-\underline{\l})\|D(u)\|+2(\l(t)\b(t)-\underline{\l}\underline{\b})\|B(u)\|\right)dt  \\
& \ +\int_0^b \left(2(\l(t)-\underline{\l})\|A_{\underline{\l}}(u-\underline{\l} D(u)-\underline{\l}\underline{\b}B(u))\|+2\|J_{\underline{\l} A}(u-\underline{\l} D(u)-\underline{\l}\underline{\b}B(u))\|\right)dt.
\end{align*}

Hence, $F(\cdot,u,v)\in L^1_{loc}([0,+\infty),{\mathcal H}\times {\mathcal H})$ for all $u,v\in {\mathcal H}.$ 
The conclusion of the theorem follows by applying the Cauchy-Lipschitz-Picard theorem to the first order dynamical system \eqref{firstorder}.

\end{proof}

\section{Some preparatory lemmas}\label{sec3}

In this section we provide some preparatory lemmas which will be used when proving the convergence of the trajectories generated by the dynamical system \eqref{e5}. We start by recalling two central results; see for example \cite[Lemma 5.1]{abbas-att-sv} and \cite[Lemma 5.2]{abbas-att-sv}, respectively.

\begin{lemma}\label{fejer-cont1} Suppose that $F:[0,+\infty)\rightarrow\R$ is locally absolutely continuous and bounded below and that
there exists $G\in L^1([0,+\infty))$ such that for almost every $t \in [0,+\infty)$ $$\frac{d}{dt}F(t)\leq G(t).$$
Then there exists $\lim_{t\rightarrow \infty} F(t)\in\R$.
\end{lemma}

\begin{lemma}\label{fejer-cont2}  If $1 \leq p < \infty$, $1 \leq r \leq \infty$, $F:[0,+\infty)\rightarrow[0,+\infty)$ is
locally absolutely continuous, $F\in L^p([0,+\infty))$, $G:[0,+\infty)\rightarrow\R$, $G\in  L^r([0,+\infty))$ and
for almost every $t \in [0,+\infty)$ $$\frac{d}{dt}F(t)\leq G(t),$$ then $\lim_{t\rightarrow +\infty} F(t)=0$.
\end{lemma}

\begin{lemma}\label{l1}
Suppose that $(H1)$ holds and let $x$ be the unique strong global solution of \eqref{e5}. Take $(x^*,w)\in \gr(A+D+N_C)$ such that $w=v+Dx^*+p$, where $v\in Ax^*$ and $p\in N_C(x^*).$ For every $t\ge 0$ consider the function $h(t)=\frac12\|x(t)-x^*\|^2.$ 
Then the following inequality holds for almost every $t\ge 0$: 
\begin{align}\label{e7}
& \ddot{h}(t)+\g(t)\dot{h}(t)+\l(t)\left(\frac{1}{L_D}-\l(t)\right)\|D(x(t))-Dx^*\|^2-\|\dot{x}(t)\|^2\le \nonumber\\
& \l(t)\b(t)\!\!\left(\sup_{u\in C}\varphi_B\!\!\left(u,\frac{p}{\b(t)}\right)-\sigma_C\left(\frac{p}{\b(t)}\right)\!\!\right)+\!\l^2(t)\|Dx^*+v\|^2\!+\l(t)\<w,x^*-x(t)\>\!+\frac{\l^2(t)\b^2(t)}{2}\|B(x(t))\|^2.
\end{align}
\end{lemma}

\begin{proof} We have $\dot{h}(t)=\<\dot{x}(t),x(t)-x^*\>$ and $\ddot{h}(t)=\<\ddot{x}(t),x(t)-x^*\>+\|\dot{x}(t)\|^2$ for every $t \geq 0$. By using the definition of the resolvent, the differential equation in (\ref{e5}) can be written for almost every $t \geq 0$ as
$$x(t)-\l(t) D(x(t))-\l(t)\b(t)B(x(t))\in \ddot{x}(t)+\g(t)\dot{x}(t)+x(t)+\l(t) A(\ddot{x}(t)+\g(t)\dot{x}(t)+x(t))$$
or, equivalently,
\begin{equation}\label{def-res}-\frac{1}{\l(t)}\ddot{x}(t)-\frac{\g(t)}{\l(t)}\dot{x}(t)-D(x(t))-\b(t)B(x(t))\in A(\ddot{x}(t)+\g(t)\dot{x}(t)+x(t)).\end{equation}
Since $v\in Ax^*$ and $A$ is monotone, we get for almost every $t \geq 0$
$$\left\<v+\frac{1}{\l(t)}\ddot{x}(t)+\frac{\g(t)}{\l(t)}\dot{x}(t)+D(x(t))+\b(t)B(x(t)), x^*-\ddot{x}(t)-\g(t)\dot{x}(t)-x(t)\right\>\ge0.$$
It follows that
\begin{align}\label{11prim}
\l(t)\<D(x(t))+\b(t)B(x(t))+v, x^*-\ddot{x}(t)-\g(t)\dot{x}(t)-x(t)\> & \ge \nonumber\\
\<\ddot{x}(t)+\g(t)\dot{x}(t),-x^*+\ddot{x}(t)+\g(t)\dot{x}(t)+x(t)\> = \ddot{h}(t)+\g(t)\dot{h}(t)+\|\ddot{x}(t)+\g(t)\dot{x}(t)\|^2-\|\dot{x}(t)\|^2&
\end{align}
for almost every $t \geq 0$. Hence, for almost every $t \geq 0$
\begin{align*}
\ddot{h}(t)+\g(t)\dot{h}(t)-\|\dot{x}(t)\|^2\le \l(t)\<D(x(t))+\b(t)B(x(t))+v, x^*-x(t)\>&\\
+\l(t)\<D(x(t))+\b(t)B(x(t))+v, -\ddot{x}(t)-\g(t)\dot{x}(t)\>-\|\ddot{x}(t)+\g(t)\dot{x}(t)\|^2 & \leq \\
\l(t)\<D(x(t))+\b(t)B(x(t))+v, x^*-x(t)\>+\frac{\l^2(t)}{4}\|D(x(t))+\b(t)B(x(t))+v\|^2,&
\end{align*}
and from here, by using mean inequalities,
\begin{align*}
\ddot{h}(t)+\g(t)\dot{h}(t)-\|\dot{x}(t)\|^2\le & \ \l(t)\<D(x(t))+\b(t)B(x(t))+v, x^*-x(t)\>+\\
& \ \frac{\l^2(t)\b^2(t)}{2}\|B(x(t))\|^2+\l^2(t)\|Dx^*+v\|^2+\l^2(t)\|D(x(t))-Dx^*\|^2.
\end{align*}
Since $v=w-Dx^*-p$, we obtain for the first summand of the term on the right-hand side of the above inequality for every $t \geq 0$ the following estimate
\begin{align*}
\l(t)\<D(x(t))+\b(t)B(x(t))+v, x^*-x(t)\> & = \\
\l(t)\<D(x(t))+\b(t)B(x(t))+w-Dx^*-p, x^*-x(t)\> & = \\
\l(t)\<D(x(t))-Dx^*,x^*-x(t)\>+\l(t)\<w,x^*-x(t)\>+ &\\
\l(t)\b(t)\left[\<B(x(t)),x^*\>+\left\<\frac{p}{\b(t)},x(t)\right\>-\<B(x(t)),x(t)\>-\left\<\frac{p}{\b(t)},x^*\right\>\right] &\le \\
-\frac{\l(t)}{L_D}\|D(x(t))-Dx^*\|^2+\l(t)\<w,x^*-x(t)\>+\l(t)\b(t)\left(\sup_{u\in C}\varphi_B\left(u,\frac{p}{\b(t)}\right)-\sigma_C\left(\frac{p}{\b(t)}\right)\right).&
\end{align*}
Hence, for almost every $t \geq 0$, we have
\begin{align*}
\ddot{h}(t)+\g(t)\dot{h}(t)-\|\dot{x}(t)\|^2\le & \ \frac{\l^2(t)\b^2(t)}{2}\|B(x(t))\|^2+\l^2(t)\|Dx^*+v\|^2+\l^2(t)\|D(x(t))-Dx^*\|^2\\
& \ -\frac{\l(t)}{L_D}\|D(x(t))-Dx^*\|^2+\l(t)\<w,x^*-x(t)\>\\
& \ +\l(t)\b(t)\left(\sup_{u\in C}\varphi_B\left(u,\frac{p}{\b(t)}\right)-\sigma_C\left(\frac{p}{\b(t)}\right)\right),
\end{align*}
which is nothing else than the desired conclusion.
\end{proof}

\begin{lemma}\label{l2}
Suppose that $(H1)$ holds and let $x$ be the unique strong global solution of \eqref{e5}. Take $x^*\in C\cap\dom A$ and $v\in Ax^*.$ For every $t\ge 0$ consider the function $h(t)=\frac12\|x(t)-x^*\|^2.$ 
Then for every $\e>0$ the following inequality holds for almost every $t\ge 0$:
\begin{align}\label{e8}
& \ddot{h}(t)+\g(t)\dot{h}(t)+\frac{1+2\e}{2+2\e}\|\ddot{x}(t)+\g(t)\dot{x}(t)\|^2-\|\dot{x}(t)\|^2+\frac{\e\l(t)\b(t)}{1+\e}\<B(x(t)),x(t)-x^*\>\le \nonumber\\
& \l(t)\b(t)\left(\frac{1+\e}{2}\l(t)\b(t)-\frac{1}{(1+\e)L_B}\right)\|B(x(t))\|^2+\l(t)\<D(x(t))+v,x^*-\ddot{x}(t)-\g(t)\dot{x}(t)-x(t)\>.
\end{align}
\end{lemma}

\begin{proof} Let be $\e >0$ fixed. According to \eqref{11prim} in the proof of the above lemma, we have for almost every $t \geq 0$
\begin{align*}
\ddot{h}(t)+\g(t)\dot{h}(t)+\|\ddot{x}(t)+\g(t)\dot{x}(t)\|^2-\|\dot{x}(t)\|^2\le & \ \l(t)\b(t)\<B(x(t)), x^*-x(t)\>+ \\
& \ \l(t)\b(t)\<B(x(t)), -\ddot{x}(t)-\g(t)\dot{x}(t)\>+ \\
& \ \l(t)\<D(x(t))+v, x^*-\ddot{x}(t)-\g(t)\dot{x}(t)-x(t)\>.
\end{align*}
Since $B$ is $\frac{1}{L_B}$-cocoercive and $Bx^*=0$ we have $\<B(x(t)), x^*-x(t)\>\le -\frac{1}{L_B}\|B(x(t))\|^2$, hence
$$\l(t)\b(t)\<B(x(t)), x^*-x(t)\>\le -\frac{\l(t)\b(t)}{(1+\e)L_B}\|B(x(t))\|^2+\frac{\e}{1+\e}\l(t)\b(t)\<B(x(t)), x^*-x(t)\>,$$
for every $t \geq 0$.
Consequently,
\begin{align*}
& \ddot{h}(t)+\g(t)\dot{h}(t)+\|\ddot{x}(t)+\g(t)\dot{x}(t)\|^2-\|\dot{x}(t)\|^2+\frac{\e\l(t)\b(t)}{1+\e}\<B(x(t)),x(t)-x^*\>\le -\frac{\l(t)\b(t)}{(1+\e)L_B}\|B(x(t))\|^2\\ 
& +\l(t)\b(t)\<B(x(t)), -\ddot{x}(t)-\g(t)\dot{x}(t)\> +\l(t)\<D(x(t))+v, x^*-\ddot{x}(t)-\g(t)\dot{x}(t)-x(t)\>,
\end{align*}
which, combined with
$$\l(t)\b(t)\<B(x(t)), -\ddot{x}(t)-\g(t)\dot{x}(t)\>\le \frac{1}{2(1+\e)}\|\ddot{x}(t)+\g(t)\dot{x}(t)\|^2+\frac{(1+\e)\l^2(t)\b^2(t)}{2}\|B(x(t))\|^2,$$
implies for almost every $t \geq 0$ relation \eqref{e8}.
\end{proof}

\begin{lemma}\label{l3} Suppose that $(H1)$ holds and let $x$ be the unique strong global solution of \eqref{e5}. Furthermore, suppose that $\lim\sup_{t\To +\infty} \l(t)\b(t)< \frac{1}{L_B}.$ Take $x^*\in C\cap\dom A$ and $v\in Ax^*.$  For every $t\ge 0$ consider the function $h(t)=\frac12\|x(t)-x^*\|^2.$
Then there exist $a,b,c>0$ and $t_0>0$ such that for almost every $t\ge t_0$ the following inequality holds:
\begin{align}\label{e9}
& \ddot{h}(t)+\g(t)\dot{h}(t)+c\|\ddot{x}(t)+\g(t)\dot{x}(t)\|^2+a\l(t)\b(t)\Big(\<B(x(t)),x(t)-x^*\>+\|B(x(t))\|^2\Big)\le \nonumber \\
& \left(b\l^2(t)-\frac{\l(t)}{L_D}\right)\|D(x(t))-Dx^*\|^2+\l(t)\<Dx^*+v,x^*-x(t)\>+b\l^2(t)\|Dx^*+v\|^2+\|\dot{x}(t)\|^2.
\end{align}
\end{lemma}

\begin{proof} Let be $\e > 0$. According to the previous lemma, \eqref{e8} holds for almost every $t \geq 0$. We estimate the last summand in the 
right-hand side of \eqref{e8} by using the mean inequality and the cocoercieveness of $D$. For every $t \geq 0$ we obtain
\begin{align*}
\l(t)\<D(x(t))+v,x^*-\ddot{x}(t)-\g(t)\dot{x}(t)-x(t)\> & \le\\
\frac{\e}{4(1+\e)}\|\ddot{x}(t)+\g(t)\dot{x}(t)\|^2+ \frac{\l^2(t)(1+\e)}{\e}\|D(x(t))+v\|^2+\l(t)\<D(x(t))+v,x^*-x(t)\> & =\\
\frac{\e}{4(1+\e)}\|\ddot{x}(t)+\g(t)\dot{x}(t)\|^2+\frac{\l^2(t)(1+\e)}{\e}\|D(x(t))+v\|^2+ & \\
\l(t)\<D(x(t))-Dx^*,x^*-x(t)\>+\l(t)\<Dx^*+v,x^*-x(t)\> & \le\\
\frac{\e}{4(1+\e)}\|\ddot{x}(t)+\g(t)\dot{x}(t)\|^2+\frac{\l^2(t)(1+\e)}{\e}\|D(x(t))+v\|^2+ &\\
-\frac{\l(t)}{L_D}\|D(x(t))-Dx^*\|^2+\l(t)\<Dx^*+v,x^*-x(t)\>,&
\end{align*}
which, combined with $\|D(x(t))+v\|^2\le 2\|D(x(t))-Dx^*\|^2+2\|Dx^*+v\|^2$, implies 
\begin{align*}
& \l(t)\<D(x(t))+v,x^*-\ddot{x}(t)-\g(t)\dot{x}(t)-x(t)\>\le \frac{\e}{4(1+\e)}\|\ddot{x}(t)+\g(t)\dot{x}(t)\|^2+\\
&\left(\frac{2\l^2(t)(1+\e)}{\e}-\frac{\l(t)}{L_D}\right)\|D(x(t))-Dx^*\|^2+\frac{2\l^2(t)(1+\e)}{\e}\|Dx^*+v\|^2+\l(t)\<Dx^*+v,x^*-x(t)\>.
\end{align*}
Using the above estimate in \eqref{e8}, we obtain for almost every $t \geq 0$
\begin{align*}
& \ddot{h}(t)+\g(t)\dot{h}(t)+\frac{1+2\e}{2+2\e}\|\ddot{x}(t)+\g(t)\dot{x}(t)\|^2-\|\dot{x}(t)\|^2+\frac{\e\l(t)\b(t)}{1+\e}\<B(x(t)),x(t)-x^*\>\le\\
& \l(t)\b(t)\left(\frac{1+\e}{2}\l(t)\b(t)-\frac{1}{(1+\e)L_B}\right)\|B(x(t))\|^2+ \frac{\e}{4(1+\e)}\|\ddot{x}(t)+\g(t)\dot{x}(t)\|^2+\\
& \left(\frac{2\l^2(t)(1+\e)}{\e}-\frac{\l(t)}{L_D}\right)\|D(x(t))-Dx^*\|^2+\frac{2\l^2(t)(1+\e)}{\e}\|Dx^*+v\|^2+\l(t)\<Dx^*+v,x^*-x(t)\>
\end{align*}
or, equivalently
\begin{align*}
& \ddot{h}(t)+\g(t)\dot{h}(t)+\frac{2+3\e}{4(1+\e)}\|\ddot{x}(t)+\g(t)\dot{x}(t)\|^2-\|\dot{x}(t)\|^2+\frac{\e\l(t)\b(t)}{1+\e} \left(\<B(x(t)),x(t)-x^*\>+\|B(x(t))\|^2\right) \le \\
& \l(t)\b(t)\left(\frac{1+\e}{2}\l(t)\b(t)-\frac{1}{(1+\e)L_B}+\frac{\e}{1+\e}\right)\|B(x(t))\|^2+\left(\frac{2\l^2(t)(1+\e)}{\e}-\frac{\l(t)}{L_D}\right)\|D(x(t))-Dx^*\|^2\\
& +\frac{2\l^2(t)(1+\e)}{\e}\|Dx^*+v\|^2+\l(t)\<Dx^*+v,x^*-x(t)\>.
\end{align*}

Since $\lim\sup_{t\To +\infty} \l(t)\b(t)< \frac{1}{L_B}$, there exists $t_0>0$ such that
$$\frac{1+\e}{2}\l(t)\b(t)-\frac{1}{(1+\e)L_B}+\frac{\e}{1+\e}< \frac{1+\e}{2L_B}-\frac{1}{(1+\e)L_B}+\frac{\e}{1+\e}$$ for every $t\ge t_0.$ Further, we notice that
$$\frac{1+\e}{2L_B}-\frac{1}{(1+\e)L_B}+\frac{\e}{1+\e}\le 0$$ for every $\e\in\left(0,\sqrt{(1+L_B)^2+1}-(1+L_B)\right].$ By chosing $\e_0$ from this interval and defining 
$$a:=\frac{\e_0}{1+\e_0}, b:=\frac{2(1+\e_0)}{\e_0} \ \mbox{and} \ c:=\frac{2+3\e_0}{4(1+\e_0)},$$ 
the conclusion follows.
\end{proof}

\begin{remark}\label{rem8} {\rm In the proof of the above theorem, the choice $\e_0\le \sqrt{(1+L_B)^2+1}-(1+L_B)<\sqrt{2}-1$, implies that $a< 1-\frac{1}{\sqrt{2}}$ and $\frac12<c<\frac34-\frac{\sqrt{2}}{8}.$}
\end{remark}

\begin{lemma}\label{l4} Suppose that $(H1)$ holds and let $x$ be the unique strong global solution of \eqref{e5}. Furthermore, suppose that $\lim\sup_{t\To +\infty} \l(t)\b(t)< \frac{1}{L_B}$  and $\lim_{t\To+\infty}\l(t)=0$. Take $(x^*,w)\in \gr(A+D+N_C)$ such that $w=v+Dx^*+p$, where $v\in Ax^*$ and $p\in N_C(x^*).$ For every $t\ge 0$ consider the function $h(t)=\frac12\|x(t)-x^*\|^2.$
Then there exist $a,b,c >0$ and $t_1>0$ such that for almost every $t\ge t_1$ the following inequality holds:
\begin{align}\label{e10}
& \ddot{h}(t)+\g(t)\dot{h}(t)+c\|\ddot{x}(t)+\g(t)\dot{x}(t)\|^2+a\l(t)\b(t)\left(\frac{1}{2}\<B(x(t)),x(t)-x^*\>+\|B(x(t))\|^2\right)\le \nonumber \\
& \frac{a\l(t)\b(t)}{2}\left(\sup_{u\in C}\varphi_B\left(u,\frac{2p}{a\b(t)}\right)-\sigma_C\left(\frac{2p}{a\b(t)}\right)\right)+b\l^2(t)\|Dx^*+v\|^2+\l(t)\<w,x^*-x(t)\>+\|\dot{x}(t)\|^2.
\end{align}
\end{lemma}

\begin{proof} According to Lemma \ref{l3}, there exist $a,b,c>0$ and $t_0>0$ such that for almost every $t\ge t_0$ 
the inequality (\ref{e9}) holds. Since $\lim_{t\To+\infty}\l(t)=0$, there exists $t_1\ge t_0$ such that $\l(t)\le\frac{1}{bL_D}$, hence $b\l^2(t)-\frac{\l(t)}{L_D}\le 0$ for every $t\ge t_1.$ Consequently, we can 
omit for every $t\ge t_1$ the term $\left(b\l^2(t)-\frac{\l(t)}{L_D}\right)\|D(x(t))-Dx^*\|^2$ in (\ref{e9}) and obtain that the inequality
\begin{align}\label{interm}
& \ddot{h}(t)+\g(t)\dot{h}(t)+c\|\ddot{x}(t)+\g(t)\dot{x}(t)\|^2+a\l(t)\b(t)\Big(\<B(x(t)),x(t)-x^*\>+\|B(x(t))\|^2\Big)\le \nonumber \\
& \l(t)\<Dx^*+v,x^*-x(t)\>+b\l^2(t)\|Dx^*+v\|^2+\|\dot{x}(t)\|^2
\end{align}
holds  for almost every $t\ge t_1$. 

Since $Dx^*+v=w-p$, we have for every $t \geq 0$
\begin{align*}
& \frac{a\l(t)\b(t)}{2}\<B(x(t)),x^*-x(t)\> + \l(t)\<Dx^*+v,x^*-x(t)\>=\\
& \frac{a\l(t)\b(t)}{2}\left(\<B(x(t)),x^*\>+\left\<\frac{2p}{a\b(t)},x(t)\right\>-\<B(x(t)),x(t)\>-\left\<\frac{2p}{a\b(t)},x^*\right\> \right)+ \l(t)\<w,x^*-x(t)\>\le \\
& \frac{a\l(t)\b(t)}{2}\left(\sup_{u\in C}\varphi_B\left(u,\frac{2p}{a\b(t)}\right)-\sigma_C\left(\frac{2p}{a\b(t)}\right)\right)+\l(t)\<w,x^*-x(t)\>.
\end{align*}
On the other hand, \eqref{interm} can be equivalently written for almost every $t\ge t_1$ as 
\begin{align*}
& \ddot{h}(t)+\g(t)\dot{h}(t)+c\|\ddot{x}(t)+\g(t)\dot{x}(t)\|^2+a\l(t)\b(t)\Big(\frac{1}{2}\<B(x(t)),x(t)-x^*\>+\|B(x(t))\|^2\Big)\le \nonumber \\
& \frac{a\l(t)\b(t)}{2}\<B(x(t)),x^*-x(t)\> + \l(t)\<Dx^*+v,x^*-x(t)\>+b\l^2(t)\|Dx^*+v\|^2+\|\dot{x}(t)\|^2,
\end{align*}
hence
\begin{align*}
& \ddot{h}(t)+\g(t)\dot{h}(t)+c\|\ddot{x}(t)+\g(t)\dot{x}(t)\|^2+a\l(t)\b(t)\left(\frac{1}{2}\<B(x(t)),x(t)-x^*\>+\|B(x(t))\|^2\right)\le \nonumber \\
& \frac{a\l(t)\b(t)}{2}\left(\sup_{u\in C}\varphi_B\left(u,\frac{2p}{a\b(t)}\right)-\sigma_C\left(\frac{2p}{a\b(t)}\right)\right)+b\l^2(t)\|Dx^*+v\|^2+\l(t)\<w,x^*-x(t)\>+\|\dot{x}(t)\|^2.
\end{align*}
\end{proof}

\section{Main result: the convergence of the trajectories}\label{sec4}

For the proof of the convergence of the trajectories generated by the dynamical system \eqref{e5} we will utilize the following assumptions:

$(H2):  A+N_C\mbox{ is maximally monotone and } \zer(A+D+N_C)\neq\emptyset$;

$(H3):  \l\in L^2([0,+\infty))\setminus L^1([0,+\infty))\mbox{ and } \lim_{t\To+\infty}\l(t)=0$;

$(H_{fitz}): \ \mbox{For every }p\in \ran  N_C,\, \int_0^{+\infty}\l(t)\b(t)\left(\sup_{u\in C}\varphi_B\left(u,\frac{p}{\b(t)}\right)-\sigma_C\left(\frac{p}{\b(t)}\right)\right)dt<+\infty.$

\begin{remark} \label{rm11} {\rm (i) The first assumption in $(H2)$ is fulfilled when a regularity condition which ensures the maximality of the sum of two maximally monotone operators holds. This is a widely studied topic in the literature; we refer the reader to \cite{bauschke-book, borw-06, bo-van, simons} for 
such conditions, including the classical Rockafellar's condition expressed in terms of the domains of the involved operators.

(ii) With respect to $(H_{fitz})$, we would like to remind that a similar condition formulated in terms of the Fitzpatrick function has been considered for the first time in \cite{b-c-penalty-svva} in the discrete setting. Its continuous version has been introduced in \cite{BCs} and further used also in \cite{att-cab-cz}. 

This class of conditions, widely used in the context of penalization approaches, has its origin in \cite{att-cza-10}. Here, in the particular case $C=\argmin\psi$, where $\psi:{\cal H}\rightarrow\R$ is a convex and differentiable function with Lipschitz continuous gradient and such that $\min\psi=0$, the  condition

$(H)$:  For every $p\in\ran N_C$, $\int_0^{+\infty} \lambda(t)\beta(t)\left[\psi^*\left(\frac{p}{\beta(t)}\right)-\sigma_C\left(\frac{p}{\beta(t)}\right)\right]dt<+\infty$. 

has been used in the asymptotic analysis of a coupled dynamical system with multiscale aspects. The function $\psi^*:{\cal H}\rightarrow\R\cup\{+\infty\}$, $\psi^*(u)=\sup_{x\in {\cal H}}\{\langle u,x\rangle-\psi(x)\}$, denotes the Fenchel conjugate of $\psi$.

According to \cite{bausch-m-s}, it holds
\begin{equation}\label{fitzp-subdiff-ineq}
 \varphi_{\nabla\psi }(x,u)\leq \psi(x) +\psi^*(u) \ \forall (x,u)\in {\cal H}\times {\cal H}.
\end{equation}
Since $\psi(x)=0$ for $x\in C$, condition $(H_{fitz})$ applied to $B=\nabla \psi$ is fulfilled, provided that $(H)$ is fulfilled. For several particular situations where $(H_{fitz})$ is verified (in its continuous or discrete version) we refer the reader to \cite{att-cza-10, att-cza-peyp-c, att-cza-peyp-p, banert-bot-pen, peyp-12, noun-peyp}.} 
\end{remark}

For proving the convergence of the trajectories generated by the dynamical system \eqref{e5} we will also make use of the following ergodic version of the continuous Opial Lemma (see \cite[Lemma 2.3]{att-cza-10}).

\begin{lemma}\label{opial} Let $S \subseteq { {\mathcal H}}$ be a nonempty set, $x:[0,+\infty)\rightarrow{ {\mathcal H}}$ a given map and
$\lambda:[0,+\infty)\rightarrow(0,+\infty)$ such that $\int_0^{+\infty}\lambda(t)=+\infty$. Define $\tilde x :[0, +\infty)\rightarrow { {\mathcal H}}$ by
$$\tilde x(t)=\frac{1}{\int_0^t\lambda(s)ds}\int_0^t\lambda(s)x(s)ds.$$
Assume that

(i) for every $z\in S$, $\lim_{t\rightarrow+\infty}\|x(t)-z\|$ exists;

(ii) every weak sequential cluster point of the map $\tilde x$ belongs to $S$.

\noindent Then there exists $x_{\infty}\in S$ such that $w-\lim_{t\rightarrow+\infty}\tilde x(t)=x_{\infty}$.
\end{lemma}

We can state now the main result of this paper. 

\begin{theorem}\label{t2} Suppose that $(H1)-(H3)$ and $(H_{fitz})$ hold, and let $x$ be the unique strong global solution of \eqref{e5}. Furthermore, suppose that $\lim\sup_{t\To +\infty} \l(t)\b(t)< \frac{1}{L_B}$, $\g$ is locally absolutely continuous and for almost every $t \geq 0$ it holds
$\g(t)\ge \sqrt{2}$ and $\dot{\g}(t)\le 0$. Let $\tilde{x}:[0,+\infty)\To {\mathcal H}$ be defined by
$$\tilde{x}(t)=\frac{1}{\int_0^t\l(s)ds}\int_0^t\l(s)x(s)ds.$$
Then the following statements hold:
\begin{itemize}
\item[(i)] for every $x^*\in \zer(A+D+N_C),$ $\|x(t)-x^*\|$ converges as $t\To+\infty$; in addition,
$\dot{x},\ddot{x}\in L^2([0,+\infty),{\mathcal H}),$ $\l(\cdot)\b(\cdot)\|B(x(\cdot))\|^2\in L^1([0,+\infty))$, $\int_0^{+\infty}\l(t)\b(t)\<B(x(t)),x(t)-x^*\>dt<+\infty$, and $\lim_{t\to+\infty}\dot{x}(t)=\lim_{t\to+\infty}\dot{h}(t)=0,$ where $h(t)=\frac12\|x(t)-x^*\|^2;$
\item[(ii)] $\tilde{x}(t)$ converges weakly as $t\To+\infty$ to an element in $\zer(A+D+N_C)$;
\item[(iii)] if, additionally, $A$ is strongly monotone, then $x(t)$ converges strongly as $t\To+\infty$ to  the unique element of $\zer(A+D+N_C)$. 
\end{itemize}
\end{theorem}

\begin{proof}
(i)  
Let be $x^*\in\zer(A+D+N_C)$, thus $(x^*,0) \in \gr(A+D+N_C)$ and $0=v+Dx^*+p$ for $v\in Ax^*$ and $p\in N_C(x^*)$. According to Lemma \ref{l4} and Remark \ref{rem8}, there exist $a,b,c >0$, with $a< 1-\frac{1}{\sqrt{2}}$ and $\frac12<c<\frac34-\frac{\sqrt{2}}{8}$, and $t_1>0$ such that for almost every $t\ge t_1$ it holds
\begin{align*}
& \ddot{h}(t)+\g(t)\dot{h}(t)+c\|\ddot{x}(t)+\g(t)\dot{x}(t)\|^2+a\l(t)\b(t)\left(\frac{1}{2}\<B(x(t)),x(t)-x^*\>+\|B(x(t))\|^2\right)\le  \\
& \frac{a\l(t)\b(t)}{2}\left(\sup_{u\in C}\varphi_B\left(u,\frac{2p}{a\b(t)}\right)-\sigma_C\left(\frac{2p}{a\b(t)}\right)\right)+b\l^2(t)\|Dx^*+v\|^2+\|\dot{x}(t)\|^2.
\end{align*}
On the other hand, since
$$\g(t)\dot{h}(t)=\frac{d}{dt}(\g(t)h(t))-\dot{\g(t)}{h}(t)\ge \frac{d}{dt}(\g(t)h(t)),$$ it holds
\begin{align*}
\ddot{h}(t)+\g(t)\dot{h}(t)+c\|\ddot{x}(t)+\g(t)\dot{x}(t)\|^2-\|\dot{x}(t)\|^2 & \ge \\
\frac{d}{dt}\left(\dot{h}(t)+\g(t)h(t)+c\g(t)\|\dot{x}(t)\|^2\right)+(c\g^2(t)-c\dot{\g}(t)-1)\|\dot{x}(t)\|^2+c\|\ddot{x}(t)\|^2&
\end{align*}
for every $t \geq 0$.

By combining these two inequalities,  we obtain for almost every $t\ge t_1$ 
\begin{align}\label{ineq-l2}\frac{d}{dt}\left(\dot{h}(t)+\g(t)h(t)+c\g(t)\|\dot{x}(t)\|^2\right)+(c\g^2(t)-c\dot{\g}(t)-1)\|\dot{x}(t)\|^2+c\|\ddot{x}(t)\|^2+& \nonumber \\
a\frac{\l(t)\b(t)}{2}\<B(x(t)),x(t)-x^*\>+a\l(t)\b(t)\|B(x(t))\|^2 & \le \nonumber \\
\frac{a\l(t)\b(t)}{2}\left(\sup_{u\in C}\varphi_B\left(u,\frac{2p}{a\b(t)}\right)-\sigma_C\left(\frac{2p}{a\b(t)}\right)\right)+b\l^2(t)\|Dx^*+v\|^2.&
\end{align}
Since $\g(t)\ge\sqrt{2}$ and $c>\frac12$ one has $c\g^2(t)-c\dot{\g}(t)-1\ge 2c-1>0$ for almost every $t \geq 0$. 
By using that $\<B(x(t)), x(t) - x^*\>\ge \frac{1}{L_B}\|B(x(t))\|^2$ for every $t \geq 0$ and by neglecting the nonnegative terms on the left-hand side of \eqref{ineq-l2}, we get for almost every $t\ge t_1$ 
\begin{align*}\frac{d}{dt}\left(\dot{h}(t)+\g(t)h(t)+c\g(t)\|\dot{x}(t)\|^2\right) & \le \\
\frac{a\l(t)\b(t)}{2}\left(\sup_{u\in C}\varphi_B\left(u,\frac{2p}{a\b(t)}\right)-\sigma_C\left(\frac{2p}{a\b(t)}\right)\right)+b\l^2(t)\|Dx^*+v\|^2.&
\end{align*}
Further, by integration, we easily derive that there exists $M > 0$ such that for every $t\geq 0$
\begin{equation}\label{ineq}\dot{h}(t)+\g(t)h(t)+c\g(t)\|\dot{x}(t)\|^2\leq M.\end{equation}
Hence, $\dot{h}(t)+\g(t)h(t)\le M$, which leads to $\dot{h}(t)+\sqrt{2}h(t)\le M$ for all $t\ge 0.$
Consequently, $\frac{d}{dt}(h(t)e^{\sqrt{2}t})\le M e^{\sqrt{2}t}$; therefore, by integrating this inequality from $0$ to $T>0$, one obtains
$$h(T)\le \frac{M}{\sqrt{2}}-\frac{M}{\sqrt{2}}e^{\sqrt{2}(-T)}+h(0)e^{\sqrt{2}(-T)},$$
which shows  that $h$ is bounded, hence $x$ is bounded. Combining this with 
$$\langle \dot x(t),x(t)-x^*\rangle+c\sqrt{2}\|\dot{x}(t)\|^2\leq M \ \forall t \geq 0,$$
which is a consequence of \eqref{ineq}, we derive that $\dot x$ is bounded, too. 

In conclusion, $t \mapsto \dot{h}(t)+\g(t)h(t)+c\g(t)\|\dot{x}(t)\|^2$ is bounded from below. By taking into account relation \eqref{ineq-l2} and applying Lemma \ref{fejer-cont1}, we obtain
\begin{equation}\label{e-lim}\lim_{t\To+\infty}\big(\dot{h}(t)+\g(t)h(t)+c\g(t)\|\dot{x}(t)\|^2\big)\in \R\end{equation} and
$$\int_0^{+\infty}\|\dot{x}(t)\|^2dt,\,\int_0^{+\infty}\|\ddot{x}(t)\|^2dt,\, \int_0^{+\infty}\l(t)\b(t)\<B(x(t)),x(t)-x^*\>dt,\,\int_0^{+\infty}\l(t)\b(t)\|B(x(t))\|^2\in\R.$$

Since 
$$\frac{d}{dt}\left(\frac12\|\dot{x}(t)\|^2\right)=\<\ddot{x}(t),\dot{x}(t)\>\le \frac12\|\ddot{x}(t)\|^2+\frac12\|\dot{x}(t)\|^2$$
for every $t \geq 0$ and the function on the right-hand side of the above inequality belongs to  $L^1([0,+\infty))$, according to Lemma \ref{fejer-cont2} one has $\lim_{t\to+\infty}\dot{x}(t)=0.$ Further, the equality $\dot{h}(t)=\<\dot{x}(t),x(t)-x^*\>$ leads to
$-\|\dot{x}(t)\|\|x(t)-x^*\|\le \dot{h}(t)\le \|\dot{x}(t)\|\|x(t)-x^*\|$ for every $t \geq 0$. Since  $\lim_{t\to+\infty}\dot{x}(t)=0$ and $\|x(\cdot)-x^*\|$ is bounded, one obtains $\lim_{t\to+\infty}\dot{h}(t)=0.$

From $\lim_{t\To+\infty}(\dot{h}(t)+\g(t)h(t)+c\g(t)\|\dot{x}(t)\|^2)\in \R$, $\lim_{t\To+\infty}\dot{h}(t)=0$ and $\lim_{t\To+\infty}c\g(t)\|\dot{x}(t)\|^2=0$, one obtains that the limit $\lim_{t\To+\infty}\g(t)h(t)$ exists and it is a finite number. On the other hand, since the limit $\lim_{t\To+\infty} \g(t) \ge \sqrt{2}$ exists and it is a positive number, one can conclude that 
$\lim_{t\To+\infty}h(t)$ exists and it is finite. Consequently, $\|x(t)-x^*\|$ converges as $t\To+\infty.$

(ii) We show that every weak sequential limit point of $\tilde{x}$ belongs to $\zer(A+D+N_C).$ Indeed, let $x_0$ be a weak sequential limit point of $\tilde{x}$; thus, there exists a sequence $(s_n)_{n \geq 0}$ with  $s_n\To+\infty$ and $\tilde{x}(s_n)\To x_0$ as $n\To+\infty.$

Take an arbitrary $(x^*,w)\in\gr(A+D+N_C)$ with $w=v+Dx^*+p$, $v\in Ax^*$ and $p\in N_C(x^*)$. 
Since $\lim_{t\To+\infty}\l(t)=0$, there exists $t_2>0$ such that for every $t\ge t_2$ one has $\l(t)\left(\frac{1}{L_D}-\l(t)\right)\ge 0$. From (\ref{e7}) we obtain
\begin{align*}
\ddot{h}(t)+\g(t)\dot{h}(t) & \le \\
\l(t)\b(t)\left(\sup_{u\in C}\varphi_B\left(u,\frac{p}{\b(t)}\right)-\sigma_C\left(\frac{p}{\b(t)}\right)\right)+\l^2(t)\|Dx^*+v\|^2+ &\\
\frac{\l^2(t)\b^2(t)}{2}\|B(x(t))\|^2+\|\dot{x}(t)\|^2 +\l(t)\<w,x^*-x(t)\>.\\
\end{align*}
for every $t \geq t_2$. By integrating from $t_2$ to $T>t_2$, we get from here
$$\int_{t_2}^T\ddot{h}(t)+\g(t)\dot{h}(t)dt\le L+\left\<w,\left(\int_{t_2}^T\l(t)dt\right)x^*-\int_{t_2}^T\l(t)x(t)dt\right\>,$$
where
\begin{align*}
L:= & \int_{t_2}^T \l(t)\b(t)\left(\sup_{u\in C}\varphi_B\left(u,\frac{p}{\b(t)}\right)-\sigma_C\left(\frac{p}{\b(t)}\right)\right)dt\\
& +\int_{t_2}^T\left( \l^2(t)\|Dx^*+v\|^2+\frac{\l^2(t)\b^2(t)}{2}\|B(x(t))\|^2+\|\dot{x}(t)\|^2 \right)dt.
\end{align*}

Since $\g(t)\dot{h}(t)\ge \frac{d}{dt}(\g(t)h(t))$ and $\g(T)h(T)\ge 0$, we obtain
$$-\g(t_2)h(t_2)\le L+\left\<w,\left(\int_{t_2}^T\l(t)dt\right)x^*-\int_{t_2}^T\l(t)x(t)dt\right\>-\dot{h}(T)+\dot{h}(t_2),$$
hence
$$\frac{-\g(t_2)h(t_2)}{\int_0^T\l(t)dt}\le\frac{L_1 -\dot{h}(T)}{\int_0^T\l(t)dt}+\left\<w,x^*-\frac{\int_{0}^T\l(t)x(t)dt}{\int_0^T\l(t)dt}\right\>,$$
where $L_1:=L+\dot{h}(t_2)+\left\<w,\int_0^{t_2}\l(t)x(t)dt-\left(\int_0^{t_2}\l(t)dt\right)x^*\right\> \in \R.$

We choose in the above inequality $T=s_n$ for those $n$ for which $s_n > t_2$, let $n$ converge to $+\infty$ and so, by taking into account that $\int_0^{+\infty}\l(t)dt=+\infty$ and $\dot{h}$ is bounded, we obtain
$$\<w,x^*-x_0\>\ge 0.$$
Since $(x^*,w)\in \gr(A+D+N_C)$ was arbitrary  chosen, it follows that $x_0\in \zer(A+D+N_C).$ Hence, by Lemma \ref{opial}, $\tilde{x}(t)$ converges weakly as $t\To+\infty$ to an element in $\zer(A+D+N_C)$. 

(iii) Assume now that $A$ is  strongly monotone, i.e. there exists $\eta>0$ such that
$$\<u^*-v^*,x-y\>\ge\eta\|x-y\|^2,\,\mbox{ for all }(x,u^*),(y,v^*)\in \gr(A).$$
Let $x^*$ be the unique element of $\zer(A+D+N_C).$ Then $0=v+Dx^*+p$, for $v\in Ax^*$ and $p\in  N_C(x^*).$

Since $v\in Ax^*$ and $A$ is $\eta-$strongly monotone, from \eqref{def-res} we obtain for almost every $t \geq 0$
\begin{align*}
\left\<v+\frac{1}{\l(t)}\ddot{x}(t)+\frac{\g(t)}{\l(t)}\dot{x}(t)+D(x(t))+\b(t)B(x(t)),x^*-\ddot{x}(t)-\g(t)\dot{x}(t)-x(t)\right\> & \ge\\
\eta\|x^*-\ddot{x}(t)-\g(t)\dot{x}(t)-x(t)\|^2. &
\end{align*}
By repeating the arguments in the proof of Lemma \ref{l1}, we easily derive for almost every $t \geq 0$
\begin{align*}
\ddot{h}(t)+\g(t)\dot{h}(t)-\|\dot{x}(t)\|^2+\eta\l(t)\|x^*-\ddot{x}(t)-\g(t)\dot{x}(t)-x(t)\|^2 & \le\\
\left(\l^2(t)-\frac{\l(t)}{L_D}\right)\|D(x(t))-Dx^*\|^2+\l(t)\b(t)\left(\sup_{u\in C}\varphi_B\left(u,\frac{p}{\b(t)}\right)-\sigma_C\left(\frac{p}{\b(t)}\right)\right)+ &\\
\frac{\l^2(t)\b^2(t)}{2}\|B(x(t))\|^2+\l^2(t)\|v+Dx^*\|^2. &
\end{align*}
Since $\lim_{t\To+\infty}\l(t)=0$, there exists $t_3>0$ such that $\l^2(t)-\frac{\l(t)}{L_D}\le 0$ for every $t\ge t_3$. Thus for almost every $t\geq t_3$ 
\begin{align*}
\ddot{h}(t)+\g(t)\dot{h}(t) +\eta\l(t)\|x^*-\ddot{x}(t)-\g(t)\dot{x}(t)-x(t)\|^2 & \le\\
\l(t)\b(t)\left(\sup_{u\in C}\varphi_B\left(u,\frac{p}{\b(t)}\right)-\sigma_C\left(\frac{p}{\b(t)}\right)\right)+ 
\frac{\l^2(t)\b^2(t)}{2}\|B(x(t))\|^2+\l^2(t)\|v+Dx^*\|^2 + \|\dot{x}(t)\|^2. &
\end{align*}
Combining this inequality with
$$\|x^*-x(t)\|^2\leq 2\|\ddot x(t)+\gamma(t)\dot x(t)\|^2+2\|x^*-x(t)-\ddot x(t)-\gamma(t)\dot x(t)\|^2,$$
we derive for almost every $t\geq t_3$ 
\begin{align*}
\ddot{h}(t) + \g(t)\dot{h}(t)  + \frac{\eta\l(t)}{2}\|x^*-x(t)\|^2 & \le\\
\eta\l(t)\|\ddot{x}(t)+\g(t)\dot{x}(t)\|^2 + &\\ 
\l(t)\b(t)\left(\sup_{u\in C}\varphi_B\left(u,\frac{p}{\b(t)}\right)-\sigma_C\left(\frac{p}{\b(t)}\right)\right)+\frac{\l^2(t)\b^2(t)}{2}\|B(x(t))\|^2+\l^2(t)\|v+Dx^*\|^2 + \|\dot{x}(t)\|^2. &
\end{align*}

By using that $\gamma (t)\dot h(t)\geq \frac{d}{dt}(\gamma(t)h(t))$ and $\|\ddot{x}(t)+\g(t)\dot{x}(t)\|^2\le 2\|\ddot{x}(t)\|^2+2\g^2(t)\|\dot{x}(t)\|^2$ for every $t \geq 0$, by integrating from $t_3$ to $T > t_3$, by letting $T$ converge to $+\infty$, and by using (i),  we obtain 
$$\int_0^{+\infty}\l(t)\|x^*-x(t)\|^2 dt<+\infty.$$
Since $\l\in L^2([0,+\infty))\setminus L^1([0,+\infty))$, it follows that $\|x(t)-x^*\|\To 0$ as $t\To+\infty.$
\end{proof}

\begin{remark} {\rm We want to emphasize that unlike other papers  addressing the asymptotic analysis of second order dynamical systems, and where the damping function $\gamma(t)$ was assumed to be strictly greater than $\sqrt{2}$ for all $t \geq 0$ (see \cite{att-alvarez, b-c-dyn-sec-ord}), in Theorem \ref{t2} we allow for it take this value, too.} 
\end{remark}

We close the paper by  formulating Theorem \ref{t2} in the context of the optimization problem \eqref{e1}, where we also use the relation between the assumptions $(H)$ and $(H_{fitz})$ (see Remark \ref{rm11}).

\begin{corollary}\label{c3} Consider the optimization problem \eqref{e1}. Suppose that its system of optimality conditions
$$0 \in \partial f(x) + \nabla g(x) + N_{\zer \nabla \psi}(x) $$
is solvable, that $(H1)-(H3)$ and $(H)$ hold, and let $x$ be the unique strong global solution of \eqref{e6}. Furthermore, suppose that $\lim\sup_{t\To +\infty} \l(t)\b(t)< \frac{1}{L_B}$, $\g$ is locally absolutely continuous and for almost every $t \geq 0$ it holds
$\g(t)\ge \sqrt{2}$ and $\dot{\g}(t)\le 0$. Let $\tilde{x}:[0,+\infty)\To {\mathcal H}$ be defined by
$$\tilde{x}(t)=\frac{1}{\int_0^t\l(s)ds}\int_0^t\l(s)x(s)ds.$$
Then the following statements hold:
\begin{itemize}
\item[(i)] for every $x^*\in \zer(\partial f + \nabla g + N_{\zer \nabla \psi}),$ $\|x(t)-x^*\|$ converges as $t\To+\infty$; in addition,
$\dot{x},\ddot{x}\in L^2([0,+\infty),{\mathcal H}),$ $\l(\cdot)\b(\cdot)\|\nabla \psi (x(\cdot))\|^2\in L^1([0,+\infty))$, $\int_0^{+\infty}\l(t)\b(t)\<\nabla \psi(x(t)),x(t)-x^*\>dt<+\infty$, and $\lim_{t\to+\infty}\dot{x}(t)=\lim_{t\to+\infty}\dot{h}(t)=0,$ where $h(t)=\frac12\|x(t)-x^*\|^2;$
\item[(ii)] $\tilde{x}(t)$ converges weakly as $t\To+\infty$ to an element in $\zer(\partial f + \nabla g + N_{\zer \nabla \psi})$, which is also an optimal solution of \eqref{e1};
\item[(iii)] if, additionally, $f$ is strongly convex, then $x(t)$ converges strongly as $t\To+\infty$ to  the unique element of $\zer(\partial f + \nabla g + N_{\zer \nabla \psi})$, which is the unique optimal solution of \eqref{e1}. 
\end{itemize}
\end{corollary}


\begin{thebibliography}{99}

\bibitem{abbas-att-arx14} B. Abbas, H. Attouch, {\it Dynamical systems and forward-backward algorithms associated 
with the sum of a convex subdifferential and a monotone cocoercive operator}, Optimization 64(10), 2223--2252, 2015

\bibitem{abbas-att-sv} B. Abbas, H. Attouch, B.F. Svaiter, {\it Newton-like dynamics and forward-backward methods for
structured monotone inclusions in Hilbert spaces}, Journal of Optimization Theory and its Applications 161(2), 331--360, 2014

\bibitem{alv-att-bolte-red} F. Alvarez, H. Attouch, J. Bolte, P. Redont, {\it A second-order gradient-like dissipative dynamical system
with Hessian-driven damping. Application to optimization and mechanics}, Journal de Math\'{e}matiques Pures et Appliqu\'{e}es (9) 81(8),
747-779, 2002

\bibitem{att-alvarez} H. Attouch, F. Alvarez, {\it The heavy ball with friction dynamical system for convex constrained minimization problems}, in: Optimization (Namur, 1998), Lecture Notes in Economics and
Mathematical Systems 481, Springer, Berlin, 25--35, 2000

\bibitem{att-cab-cz} H. Attouch, A. Cabot, M.-O. Czarnecki, {\it Asymptotic behavior of nonautonomous monotone and
subgradient evolution equations}, to appear in Transactions of the American Mathematical Society, arXiv:1601.00767, 2016

\bibitem{att-cza-10} H. Attouch, M.-O. Czarnecki, {\it Asymptotic behavior of coupled dynamical systems with multiscale aspects},
Journal of Differential Equations 248(6), 1315--1344, 2010

\bibitem{att-cza-16} H. Attouch, M.-O. Czarnecki, {\it Asymptotic behavior of gradient-like dynamical systems involving 
inertia and multiscale aspects},  arXiv:1602.00232, 2016

\bibitem{att-cza-peyp-p} H. Attouch, M.-O. Czarnecki, J. Peypouquet, {\it Prox-penalization and splitting methods for 
constrained variational problems}, SIAM Journal on Optimization 21(1), 149--173, 2011

\bibitem{att-cza-peyp-c} H. Attouch, M.-O. Czarnecki, J. Peypouquet, {\it Coupling forward-backward with penalty 
schemes and parallel splitting for constrained variational inequalities}, SIAM Journal on Optimization 21(4), 1251--1274, 2011

\bibitem{banert-bot-pen} S. Banert, R.I. Bo\c t, {\it Backward penalty schemes for monotone inclusion problems}, 
Journal of Optimization Theory and Applications 166(3), 930--948, 2015

\bibitem{bauschke-book} H.H. Bauschke, P.L. Combettes, {\it Convex Analysis and Monotone Operator Theory in Hilbert Spaces}, CMS Books in Mathematics, Springer, New York, 2011

\bibitem{bausch-m-s} H.H. Bauschke, D.A. McLaren, H.S. Sendov, {\it Fitzpatrick functions: inequalities, examples and
remarks on a problem by S. Fitzpatrick}, Journal of Convex Analysis 13(3-4), 499--523, 2006

\bibitem{borw-06} J.M. Borwein, {\it Maximal monotonicity via
convex analysis}, Journal of Convex Analysis 13(3-4), 561--586, 2006

\bibitem{bo-van} J.M. Borwein, J.D. Vanderwerff, {\it Convex Functions: Constructions, Characterizations
and Counterexamples}, Cambridge University Press, Cambridge, 2010

\bibitem{b-c-penalty-svva} R.I. Bo\c t, E.R. Csetnek, {\it Forward-backward and Tseng's type penalty schemes for monotone inclusion problems}, 
Set-Valued and Variational Analysis 22, 313--331, 2014

\bibitem{BCs} R.I. Bo\c t, E. R. Csetnek, {\it Approaching the solving of constrained variational inequalities via 
penalty term-based dynamical systems}, Journal of Mathematical Analysis and Applications 435, 1688--1700, 2016

\bibitem{b-c-dyn-KM} R.I. Bo\c t, E.R. Csetnek, {\it A dynamical system associated with the fixed points set of a 
nonexpansive operator}, Journal of Dynamics and Differential Equations, DOI: 10.1007/s10884-015-9438-x, 2015

\bibitem{b-c-aa} R.I. Bo\c t, E.R. Csetnek, {\it Second order dynamical systems associated to variational inequalities}, 
Applicable Analysis, http://dx.doi.org/10.1080/00036811.2016.1157589

\bibitem{b-c-dyn-sec-ord} R.I. Bo\c t, E.R. Csetnek, {\it Second order forward-backward dynamical systems for
monotone inclusion problems}, SIAM Journal on Control and Optimization 54(3), 1423--1443, 2016

\bibitem{brezis} H. Br\'{e}zis, {\it Op\'{e}rateurs maximaux monotones et semi-groupes de contractions dans les espaces de Hilbert}, 
North-Holland Mathematics Studies No. 5, Notas de Matem\'{a}tica (50), North-Holland/Elsevier, New York, 1973

\bibitem{fitz} S. Fitzpatrick, {\it Representing monotone
operators by convex functions}, in: Workshop/Miniconference on
Functional Analysis and Optimization (Canberra, 1988), Proceedings
of the Centre for Mathematical Analysis 20, Australian National
University, Canberra, 59--65, 1988

\bibitem{haraux} A. Haraux, {\it Syst\`{e}mes Dynamiques Dissipatifs et Applications},
Recherches en Math\'{e}- matiques Appliqu\'{e}ées 17, Masson, Paris, 1991

\bibitem{Ne} Y. Nesterov, Introductory Lectures on Convex Optimization: A Basic Course,
Kluwer Academic Publishers, Dordrecht, 2004

\bibitem{noun-peyp} N. Noun, J. Peypouquet, {\it Forward-backward penalty scheme for constrained convex minimization 
without inf-compactness}, Journal of Optimization Theory and Applications, 158(3), 787--795, 2013

\bibitem{peyp-12} J. Peypouquet, {\it Coupling the gradient method with a general exterior penalization scheme for 
convex minimization}, Journal of Optimizaton  Theory and Applications 153(1), 123--138, 2012

\bibitem{simons} S. Simons, {\it From Hahn-Banach to Monotonicity}, Springer, Berlin, 2008

\bibitem{sontag} E.D. Sontag, {\it Mathematical control theory. Deterministic finite-dimensional systems}, Second edition,
Texts in Applied Mathematics 6, Springer-Verlag, New York, 1998

\bibitem{su-boyd-candes} W. Su, S. Boyd, E.J. Candes, {\it A differential equation for modeling Nesterov's 
accelerated gradient method: theory and insights}, arXiv:1503.01243, 2015

\end{thebibliography}
\end{document}